\documentclass[11pt,reqno]{amsart}
\usepackage{a4wide}
\usepackage{amssymb}
\usepackage{amsthm}
\usepackage{amsmath,todonotes}
\usepackage{amscd}
\usepackage{cancel}
\usepackage{verbatim}
\usepackage{color}
\usepackage[all]{xy}
\usepackage[mathscr]{eucal}
\usepackage{mathrsfs}
\usepackage{fullpage}
\usepackage{enumitem}
\usepackage{hyperref}
\usepackage{bbm}
\usepackage{qtree}
\usepackage{bm}
\usepackage{bigints}
\allowdisplaybreaks

\numberwithin{equation}{section}

\newtheorem{theorem}{Theorem}

\newtheorem{lemma}[theorem]{Lemma}

\newtheorem{corollary}[theorem]{Corollary}

\newtheorem{proposition}[theorem]{Proposition}

\theoremstyle{remark}

\newtheorem*{remark}{Remark}

\theoremstyle{definition}

\numberwithin{theorem}{section} 
\numberwithin{equation}{section}
\numberwithin{table}{section}
\newtheorem{thma}{Theorem}

\newcommand{\R}{\mathbb{R}}
\newcommand{\C}{\mathbb{C}}
\newcommand{\K}{\mathbb{K}}

\newcommand{\Q}{\mathbb{Q}}

\DeclareMathOperator{\Res}{Res}

\renewcommand{\(}{\left(}
\renewcommand{\)}{\right)}

\renewcommand{\[}{\left[}
\renewcommand{\]}{\right]}

\begin{document}
\title[Piltz divisor problem over number fields \`a la Vorono\"i]{Piltz divisor problem over number fields \`a la Vorono\"i}
\author{Soumyarup Banerjee}
\address{\rm Department of Mathematics, University of Hong Kong, Pokfulam, Hong Kong}
\email{soumya.tatan@gmail.com}
\subjclass[2010] {11R42, 11R11, 11S40, 33C60}

\keywords{Piltz divisor problem, Dedekind zeta function, Special function, Riesz sum}

\medskip
\begin{abstract}
In this article, we study the {\rm Piltz divisor problem}, which is sometimes called the
{\rm generalized Dirichlet divisor problem}, over number fields. We establish an
identity akin to Vorono\"i's formula concerning the error term in the Dirichlet divisor problem. 
\end{abstract}

\thanks{The research of the author was supported by grants from the Research Grants Council of the Hong Kong SAR, China.}
\maketitle

\section{Introduction}
The asymptotic behavior of arithmetic functions has long been a fascinating subject in analytic number theory. In particular, one often investigates the behaviour of $a(n)$ as $n$ increases. A common technique to understand arithmetic functions involves studying the partial sums $\sum_{n\leq x}a(n)$. For example, Dirichlet famously estimated the asymptotic behaviour of the partial sums $\sum_{n\leq x} d(n)$ by relating it to the problem of counting the number of lattice points lying inside or on the hyperbola. 
Here $d(n)$ denotes the divisor function i.e, $d(n) = \sum_{d\mid n} 1$. He obtained an asymptotic formula with the main term $x\log x + (2\gamma - 1)x + \frac{1}{4}$, where $\gamma$ is the Euler-Mascheroni constant and an error term of order $\sqrt{x}$. The problem of estimating the error term between the sum $\sum_{n\leq x} d(n)$ and the main term is known as the Dirichlet hyperbola problem or the Dirichlet divisor problem. The bound on the error term has been further improved by many mathematicians.

Vorono\"i \cite{Voronoi} introduced a new phase into the Dirichlet divisor problem. He was able to express the error term as an infinite series containing the Bessel functions. 
More precisely, letting $Y_1$ (resp. $K_1$) denote the Bessel function of the second kind (resp. modified Bessel function of second kind) and $\gamma$ denote the Euler-Mascheroni constant, Vorono\"i obtained the following.
\begin{thma}[Vorono\"i identity]
For every $x > 0$, we have
\begin{equation*}
\sideset{}{'}\sum_{n\le x}\!\!d(n) =x\log x+(2\gamma-1)x+\frac{1}{4} - \sum_{k=1}^\infty\frac{d(k)}{k}
\left(Y_1\left(4\pi\,\sqrt{xk}\,\right)+\frac{2}{\pi}K_1\left(4\pi\,\sqrt{xk}\,\right)\right)\sqrt{xk},
\end{equation*}
where $\sum'$ means that the term corresponding to $n=x$ is halved.
\end{thma}
A very natural generalization of the Dirichlet divisor problem is the determination of asymptotics for the partial sums $\sum_{n\leq x} d_k(n)$ where $d_k(n)$ counts the number of ways that $n$ can be written as a product of $k$ numbers. The problem of estimating the error term is known as the Piltz divisor problem, named in honor of 
of Adolf Piltz. An error term of Vorono\"i-type was previously obtained by the author and Wang \cite{Wang} for the shifted Piltz divisor problem; this is the problem of counting the number of lattice points lying inside or on the hyperbola after shifting the origin to a fixed coordinate.  In this article, we consider the Piltz divisor problem over number fields, which we next describe. Let $\K$ be a number field with extension degree $[\K : \Q] = d$ and signature $(r_1, r_2)$ (i. e., $d = r_1+2r_2$). We let $D_\K$ denote the absolute value of the discriminant of $\K$. Let $\mathcal{O}_\K$ be its ring of integers and $v_{\K}(m)$ denote the number of non-zero integral ideals in $\mathcal{O}_K$ with norm $m$. Let $\mathfrak{N}$ be the norm map of $\K$ over $\Q$. The main question considered in this paper is the problem of counting the number of $m$-tuples of ideals $(\mathfrak{a}_1, \mathfrak{a}_2, \ldots, \mathfrak{a}_m)$ for which the product of their ideal norms
 $\mathfrak{N}_{\mathfrak{a}_1}\mathfrak{N}_{\mathfrak{a}_2}\cdots \mathfrak{N}_{\mathfrak{a}_m}$ is at most $x$. In other words, we study the asymptotic behaviour of the partial sum
$$
I_\K^m(x) = \sideset{}{'}\sum_{n \leq x} v_{\K}^m(n)
$$
where $v_{\K}^m(n)$  counts the number of m-tuples of ideals $(\mathfrak{a}_1, \mathfrak{a}_2, \ldots, \mathfrak{a}_m)$ with $\mathfrak{N}_{\mathfrak{a}_1}\mathfrak{N}_{\mathfrak{a}_2}\cdots \mathfrak{N}_{\mathfrak{a}_m} = n$. We denote the main term and the error term of $I_\K^m(x)$ by  $P_\K^m(x)$ and $\Delta_\K^m(x)$, 
respectively. It is known that the main term is $P_\K^m(x) = \Res_{s=1} \zeta_{\K}(s)^m x^s s^{-1}$, where $\zeta_\K(s)$ denotes the Dedekind zeta function of the number field $\K$. The main term can be obtained by a standard procedure using complex analysis. The main goal of this article is to estimate the error term $\Delta_\K^m(x)$ in terms of special functions.

The asymptotic behaviour of the partial sum $I_\Q^m(x)$ is well studied, with many mathematicians having already investigated this sum as early as the 18th century. The best known result so far for $\Delta_\Q^2(x)$ is $x^{517/1648+\epsilon}$ \cite{Bourgain}. The partial sum $I_\K^m(x)$ with $m=1$ is the ordinary ideal counting function over $\K$ and there are numerous investigations of the asymptotics for these dating back to the 19th century. The best known result can be found in \cite{Takeda} 
or \cite{Krishnarjun}. In comparison to other divisor problems, it seems that very few improvements [c.f \cite{Nowak} \cite{Takeda}, \cite{Krishnarjun}] have been made on the 
Piltz divisor problem over number fields. In this article, we obtain a Vorono\"i-type identity for Piltz divisor problem over number fields. In particular, we express the error term in terms of an infinite series containing the ``Meijer G-function". Moreover, we discuss certain special cases in another section where the error terms may be written 
in terms of different special functions.  
\begin{theorem}\label{Theorem 1}
Let $\K$ be any number field of degree $d$ with signature $(r_1, r_2)$ and discriminant $d_\K$ {\rm(}with $|d_{\K}|=D_{\K}${\rm)}. Then
\begin{align}
I_\K^m(x) =\zeta_\K(0)^m + H_m x &- \frac{i^{mr_1}D_\K^{m/2}}{(2\pi)^{mr_1+mr_2}}\sum\limits_{j=0}^{mr_1}(-1)^j                \(\begin{matrix}
mr_1\\ j
\end{matrix}\)
\sum\limits_{n=1}^\infty \frac{v_\K^m(n)}{n} \nonumber\\ 
&\times G_{0, \,\ md}^{m(r_1+r_2), 0}\(\begin{matrix}
- \\  \boldsymbol{1}_{mr_1+mr_2-1}, 0, \boldsymbol{1}_{mr_2}
\end{matrix}
\bigg| (e^{\frac{i\pi}{2}})^{2j-mr_1} \frac{(2\pi)^{dm}}{D_{\K}^m} nx\),
\end{align}
where $\boldsymbol{1}_\ell$ denotes the $\ell$-tuple all of whose entries equal $1$ and $H_m =\underset{s=1} \Res \, \ \zeta_\K(s)^m$.
\end{theorem}
Steen introduced a new function in \cite{Steen} that naturally appears when investigating the general divisor problem; we call this function the Vorono\"i-Steen function, and its definition may be found in \S \ref{sec:specialfunctions}. It reduces to a modified Bessel function of the second kind as a special case. In the special case that the number field $\K$ is totally real, the error term can be expressed as an infinite series containing the Vorono\"i-Steen function.
\begin{corollary}\label{Corollary 1.2}
Let $\K$ be any totally real number field of degree $d$ with discriminant $d_\K$. Then
\begin{equation}
I_\K^m(x) =\zeta_\K(0)^m + H_m x - \frac{i^{md}D_\K^{m/2}}{(2\pi)^{md}}\sum\limits_{j=0}^{md}(-1)^j \binom{md}{j}\sum\limits_{n=1}^\infty \frac{v_\K^m(n)}{n}  
 \, V\( (e^{\frac{i\pi}{2}})^{2j-md} \frac{(2\pi)^{dm}}{D_{\K}^m} nx ; \boldsymbol{1}_{md-1}, 0 \).
\end{equation}
\end{corollary}
We next consider the special case where the number field is purely imaginary.
\begin{corollary}\label{Corollary 1.3}
Let $\K$ be any purely imaginary number field of degree $d$ with discriminant $d_\K$. Then
\begin{align}
I_\K^m(x) =\zeta_\K(0)^m + H_m x - \(\frac{D_{\K}}{(2\pi)^d}\)^{m/2}
\sum\limits_{n=1}^\infty \frac{v_\K^m(n)}{n} 
G_{0, \,\ md}^{\frac{md}{2}, 0}\(\begin{matrix}
- \\  \boldsymbol{1}_{\frac{md}{2}-1}, 0, \boldsymbol{1}_{\frac{md}{2}}
\end{matrix}
\bigg| \frac{(2\pi)^{dm}}{D_{\K}^m} nx\).
\end{align}
\end{corollary}

The paper is organised as follows. In \S 2, we discuss the main ingredients that are needed to prove our results. In \S 3, we provide the proof of Theorem \ref{Theorem 1} and the corollaries. In \S 4, we discuss special cases of Theorem \ref{Theorem 1}.

\section{Preliminaries}
Throughout the paper, we require some basic tools of analytic number theory and complex analysis. 

\subsection{Gamma function}
The Gamma function plays a significant role in this paper. For $\mathfrak{R}(z) > 0$, it
can be defined  via the convergent improper integral 
\begin{equation}
\Gamma(z) = \int_0^\infty e^{-t} t^{z-1} {\rm d}t.
\end{equation}
The analytic properties and functional equation of the $\Gamma$-function are given in the following proposition.
\begin{proposition}
The function $\Gamma(z)$ is absolutely convergent for $\mathfrak{R}(z) > 0$. It can be analytically continued to the whole complex plane except for simple poles at every non-positive integers. It also satisfies the functional equation:
\begin{equation}\label{Gamma functional equation}
\Gamma(z+1) = z\Gamma(z).
\end{equation}
\end{proposition}
\begin{proof}
This is well known, and a proof may be found, for example, in \cite[Appendix A]{Ayoub}.
\end{proof}
The $\Gamma$-function satisfies many important properties. Here we mention two of them.
\begin{itemize}
\item[(i)]
 Euler's reflection formula : 
 \begin{equation}\label{Reflection formula}
 \Gamma(z)\Gamma(1-z) = \frac{\pi}{\sin \pi z}
 \end{equation}
 where $z \notin \mathbb{Z}$.
 \item[(ii)]
  Legendre duplication formula :
  \begin{equation}\label{Duplication formula}
  \Gamma(z)\Gamma \left(z+\frac{1}{2}\right) = 2^{1-2z} \sqrt{\pi} \Gamma(2z).
  \end{equation}
\end{itemize}
The proof of these properties can be found in \cite[Appendix A]{Ayoub}.

\subsection{Dedekind zeta function}
Let $\K$ be any number field with extension degree $[\K : \Q] = d$ and signature $(r_1, r_2)$ (i. e., $d = r_1+2r_2$) and $D_\K$ denotes the absolute value of the discriminant of $\K$. Let $\mathcal{O}_\K$ be its ring of integers and $\mathfrak{N}$ be the norm map of $\K$ over $\Q$. Then the \begin{it}Dedekind zeta function\end{it} attached to number field $\K$ is defined by 
$$
\zeta_\K(s)=\sum_{\mathfrak{a}\subset\mathcal{O}_\K}\frac{1}{N(\mathfrak{a})^s}=\prod_{\mathfrak{p}\subset \mathcal{O}_\K}\bigg(1-\frac{1}{N(\mathfrak{p})^s}\bigg)^{-1},
$$
for all $s \in \C$ with $\mathfrak{R} (s)>1$, where $\mathfrak{a}$ and $\mathfrak{p}$ run over the non-zero integral ideals and prime ideals of $\mathcal{O}_\K$ respectively.
If $v_\K(m)$ denotes the number of non-zero integral ideals in $\mathcal{O}_\K$ with norm $m$, then $\zeta_\K$ can also be expressed as 
$$
\zeta_\K(s)=\sum_{m=1}^\infty \frac{v_\K(m)}{m^s}.
$$
Set
\begin{equation*}
\Lambda_\K(s) = D_\K^{s/2}\Gamma_\R(s)^{r_1} \Gamma_\C(s)^{r_2}\zeta_\K(s), 
\end{equation*}
where $\Gamma_\R(s)  = \pi^{-s/2}\Gamma(s/2)$ and $\Gamma_\C(s) = 2(2\pi)^{-s}\Gamma(s)$.
The following proposition provides the analytic behaviour and the functional equation satisfied by the Dedekind zeta function.
\begin{proposition}\label{Prop 2.2}
The function $\Lambda_\K(s)$ is absolutely convergent for $\mathfrak{R}(s) > 1$. It can be analytically continued to the whole complex plane except for a simple pole at $s = 1$. It also satisfies the functional equation
\begin{equation}\label{functional equation}
\Lambda_\K(s) = \Lambda_K(1-s).
\end{equation}
\end{proposition}
\begin{proof}
For example, one can find a proof of this statement in \cite[pp. 254-255]{Lang}.
\end{proof}
The following lemma gives the convexity bound of the Dedekind zeta function in the critical region.
\begin{lemma}\label{Convexity bound}
Let $\K$ be any number field of degree d with discriminant $d_\K$ (so $D_{\K}=|d_{\K}|$). Then
\begin{equation*}
\zeta_\K(\sigma+it) \ll \begin{cases}
|t|^{\frac{d}{2}-d\sigma+\epsilon} D_\K^{\frac{1}{2}-\sigma+\epsilon}, & \text{if} \, \ \sigma \leq 0, \\
|t|^{\frac{d(1-\sigma)}{2}+\epsilon} D_\K^{\frac{1-\sigma}{2}+\epsilon}, & \text{if} \, \ 0 \leq \sigma \leq 1,\\
|t|^\epsilon D_\K^\epsilon, & \text{if} \, \ \sigma \geq 1
\end{cases}
\end{equation*}
holds true for any $\epsilon>0$. 
\end{lemma}
\begin{proof}
This follows from a standard argument by applying the Phragmen-Lindel\"of principle and the functional equation \eqref{functional equation} of the Dedekind zeta function. The details may be found in \cite[Chapter 5]{Iwaniec}, for example.
\end{proof}

\subsection{Special function}\label{sec:specialfunctions}
The mathematical functions which have more or less established names and notations due to their importance in mathematical analysis, functional analysis, geometry, physics, or other applications are known as special functions. These mainly appear as solutions of differential equations or integrals of elementary functions. 

One of the most important families  of special functions are the Bessel functions,  which are basically the canonical solution of Bessel's differential equations
\begin{equation*}
x^2\frac{d^2y}{dx^2}+x\frac{dy}{dx}+(x^2-a^2)y = 0
\end{equation*}
where $a$ is any arbitrary complex number.

The G-function was introduced initially by Meijer as a very general function using a series. Later, it was defined more generally via a line integral in the complex plane (cf. \cite{Bateman}) given by 
\begin{equation}\label{G-function}
\begin{aligned}
G^{m, \ n}_{p, \ q}\bigg(\begin{matrix}
a_1, \ldots, a_p \\
b_1, \ldots, b_q
\end{matrix} \ \bigg|\ z\bigg)=\frac{1}{2\pi i}\underset{{(C)}}{\bigints} \frac{\prod\limits_{j=1}^m\Gamma(b_j-s)\prod\limits_{j=1}^n\Gamma(1-a_j+s)}{\prod\limits_{j=m+1}^q\Gamma(1-b_j+s)\prod\limits_{j=n+1}^p\Gamma(a_j-s)}z^s \rm{d}s ,
\end{aligned}
\end{equation}
where $z \neq 0$ and $m$, $n$, $p$, $q$ are integers which satisfy $0 \leq m \leq q$ and $0 \leq n \leq p$. The poles of the integrand must be all simple. Here $(C)$ in the integral denotes the vertical line from $C-i\infty$ to $C+i\infty$  such that all poles of $\Gamma(b_j-s)$ for $ j=1, \ldots, m$, must lie on one side of the vertical line while all poles of $\Gamma(1-a_j+s)$ for $ j=1, \ldots, n$ must lie on the other side. The integral then converges for $|\arg z| < \delta \pi$ where
$$\delta = m+n - \frac{1}{2}(p+q).$$
The integral additionally converges for $|\arg z|= \delta \pi$ if $(q-p)(\Re(s) + 1/2) > \Re (v) + 1$, where 
$$
v = \sum_{j=1}^{q}b_j - \sum_{j=1}^{p}a_j.
$$
Special cases of the $G$-function include many other special functions. For instance, there are many formulae which yield relations between the $G$-function and the Bessel functions (cf. \cite{Bateman}). Two important formulae among them are given by
\begin{equation}\label{G and J function}
G_{0,2}^{1,0}\( \begin{array}{cl}
- \\
a, b
\end{array} \bigg | z \) = z^{\frac{1}{2}(a+b)} J_{a-b} (2z^{1/2}),
\end{equation}
\begin{equation}\label{G and K function}
G_{0,2}^{2,0}\( \begin{array}{cl}
- \\
a, b
\end{array} \bigg | z \) = 2z^{\frac{1}{2}(a+b)} K_{a-b} (2z^{1/2}).
\end{equation}
The Vorono\"i-Steen function $V=V(x;a_1,\ldots,a_n)$
(cf. \cite{Steen}) is defined by
\begin{align*}
 \frac{1}{2\pi i}&\int_{0}^{\infty}x^sV(x;a_1,\ldots,a_n)\,\frac{{\rm d}x}{x} =\Gamma(s+a_1)\cdots \Gamma(s+a_n).
\end{align*}
It is a special case of the $G$-function:
\begin{equation}\label{Voronoi function}
V(x;a_1,\ldots,a_n)=G_{0,n}^{n,0}\!\left( \begin{array}{c} - \\ a_1,\ldots,a_n
 \end{array}  \bigg| \, \  x \right).
\end{equation}
 
\subsection{Riesz sum}\label{Riesz sum} 
Riesz sums (cf. \cite{CM}, \cite{Hardy}, \cite{KT}) were introduced by M. Riesz and have been studied in connection with summability of Fourier series and that of Dirichlet series. For a given increasing sequence $\{\lambda_n\}$ of positive real numbers and a given sequence $\{\alpha_n\}$ of complex numbers, the Riesz sum of order $\rho$ is defined by
\begin{equation}
\mathcal{A}^\rho(x)=\mathcal{A}_\lambda^\rho (x)= \sideset{}{'}\sum_{\lambda_n\leq x}(x-\lambda_n)^\rho\alpha_n,
\end{equation}
where $\rho$ is any non-negative integer and the prime appearing next to the summation sign means the corresponding term is to be halved for $\lambda_n=x$. 
It can also be expressed as 
\begin{equation}\label{Rieszsum}
\mathcal{A}_\lambda^\rho(x)= \rho\int_0^x(x-t)^{\rho-1}\mathcal{A}_\lambda (t){\rm d}t
\end{equation} 
for $\rho \geq 1$, where $\mathcal{A}_\lambda(x)=\mathcal{A}_\lambda^0(x)=\sideset{}{'}\sum_{\lambda_n\leq x}\alpha_n$ (cf. \cite{Hardy}, \cite{KT}). The  generalization of Perron's formula for the $\rho$-order Riesz sum is given by
\begin{equation}\label{Perron's formula}
\frac{1}{\Gamma(\rho+1)}\sideset{}{'}\sum\limits_{\lambda_n\leq x}(x-\lambda_n)^\rho \alpha_n=\frac{1}{2\pi i}\int_{C-i\infty}^{C+i\infty}\frac{\Gamma(w)\varphi(w)x^{\rho+w}}{\Gamma(w+\rho+1)}{\rm d}w,
\end{equation}
where $\varphi(w)=\sum\limits^\infty_{n=1}\frac{\alpha_n}{\lambda_n^w}$ and $C$ is bigger than the abscissa of absolute convergence of $\varphi(s)$ (cf. \cite{Hardy}, \cite{KT}). 
\begin{remark}\label{Remark}
Note that the integral in \eqref{Perron's formula} is an improper integral for the unbounded region. Hence it is not obvious that one can interchange the integral and the summation which is coming from the Dirichlet series $\varphi(s)$. Moreover, it is warned in \cite{Davenport} that applying the 0th order Perron's formula is problematic for this reason. It is usually safer to apply the truncated Perron's formula which can be found in many textbooks. The integral
\begin{align*}
\int_{(C)}\bigg | \frac{\Gamma(w)\varphi(w)x^{\rho+w}}{\Gamma(w+\rho+1)}\bigg |{\rm d}w &\ll x^{C+\rho} \int_{-\infty}^{\infty}\frac{1}{(C^2 +t^2)^{(\rho+1)/2}}\,{\rm d}t
\ll x^{C+\rho}\int_0^{\infty}\frac{1}{(C^2 +t^2)^{(\rho+1)/2}}\,{\rm d}t\\
&\ll x^{C+\rho}\[\int_0^1 \frac{1}{(C^2 +t^2)^{(\rho+1)/2}}\,{\rm d}t + \lim_{R \to \infty} \int_1^R \frac{1}{t^{\rho+1}} \,{\rm d}t \] 
 < \infty
\end{align*}
for $\rho \geq 1$. Hence it follows from Fubini's theorem that the integral and summation appearing in the $\rho$-order Riesz sum can be interchanged for $\rho \geq 1$.
\end{remark}
The lower-order Riesz sums can be obtained  from the higher-order Riesz sums by using the following lemma. 
\begin{lemma}\label{Lemma Riesz sum}
Let $\mathcal{A}_\lambda^\rho$ be the Riesz sum of order $\rho$ where $\rho$ is any non-negative integer. Then
\begin{equation}
\frac{d^i}{dx^i}\mathcal{A}_\lambda^\rho(x) =\rho(\rho-1)\cdots(\rho-i+1)\mathcal{A}_\lambda^{\rho-i}(x)
\end{equation}
holds true for every $0 \leq i \leq \rho$. In particular, we have
\begin{equation}
\frac{d^\rho}{dx^\rho}\(\frac{1}{\Gamma(\rho + 1)}\mathcal{A}_\lambda^\rho(x) \)= \mathcal{A}_\lambda(x).
\end{equation}
\end{lemma}
\begin{proof}
We will prove this lemma by induction on $\rho$. The statement holds trivially for the case $\rho = 0$ and the case of $\rho = 1$ follows from \eqref{Rieszsum}. Let us assume that the statement is true for all $2 \leq \rho \leq k$. We need to show that the statement is true for $\rho = k+1$.

It follows from \eqref{Rieszsum} and integration by parts that
\begin{align*}
\mathcal{A}_\lambda^{k+1}(x) &= (k+1)\int_0^x(x-t)^k \mathcal{A}_\lambda (t){\rm d}t\\
&= (k+1) \[ (x-t)^k \mathcal{A}_\lambda^1(t) \]_{t=0}^{t=x} + (k+1)k\int_0^x \mathcal{A}_\lambda^1(t) (x-t)^{k-1}{\rm d}t
\end{align*}
where the first term of the right-hand side becomes $0$, since $\mathcal{A}_\lambda^j(0) = 0$ for any non-negative $j$. We have from the inductive hypothesis that 
\begin{equation}\label{Derivative}
\frac{d}{dx}\mathcal{A}_\lambda^\rho(x) = \rho \mathcal{A}_\lambda^{\rho-1}(x)
\end{equation}
for every $2 \leq \rho \leq k$. Now proceeding similarly and applying \eqref{Derivative} repeatedly, we have
\begin{align}\label{Derivative2}
\mathcal{A}_\lambda^{k+1}(x) &= (k+1)k\int_0^x \mathcal{A}_\lambda^1(t) (x-t)^{k-1}{\rm d}t \nonumber\\
& = (k+1)k \frac{k-1}{2}\int_0^x \mathcal{A}_\lambda^2(t) (x-t)^{k-2}{\rm d}t \nonumber\\
& = (k+1)k \frac{k-1}{2}\frac{k-2}{3}\int_0^x \mathcal{A}_\lambda^3(t) (x-t)^{k-3}{\rm d}t \nonumber\\
& \hspace{.2cm}  \vdots \nonumber\\
& = (k+1) \int_0^x \mathcal{A}_\lambda^k(t){\rm d}t.
\end{align}
Now applying \eqref{Derivative2} and the inductive hypothesis respectively, we can finally conclude
\begin{align}
\frac{d^i}{dx^i}\mathcal{A}_\lambda^{k+1}(x) = \frac{d^{i-1}}{dx^{i-1}}\frac{d}{dx}\mathcal{A}_\lambda^{k+1}(x) =(k+1) \frac{d^{i-1}}{dx^{i-1}} \mathcal{A}_\lambda^k(x) = (k+1)k\cdots (k-i+2) \mathcal{A}_\lambda^{k+1-i}(x)
\end{align}
for every $0 \leq i \leq k+1$, which yields the claim.
\end{proof}

\section{Proof of results}
In this section we prove Theorem \ref{Theorem 1} and the corollaries of Theorem \ref{Theorem 1}.
\subsection{Proof of Theorem \ref{Theorem 1} :}
Consider the Riesz sum $I_\K^{m, \rho}(x)$ of order positive integer  $\rho$ with 
\begin{equation}\label{Assumption order}
\rho \geq \frac{md}{2}(1-\mu) + 1
\end{equation}
where $-1 < \mu < 0$ and normalized by the $\Gamma$-factor as via
\begin{equation}
I_\K^{m, \rho}(x) =\frac{1}{\Gamma(\rho+1)} \sideset{}{'}\sum_{n \leq x} (x-n)^\rho v_{\K}^m(n),
\end{equation}
where $ v_{\K}^m(n)$  counts the number of m-tuples of ideals $(\mathfrak{a}_1, \mathfrak{a}_2, \ldots, \mathfrak{a}_m)$ with $\mathfrak{N}_{\mathfrak{a}_1}\mathfrak{N}_{\mathfrak{a}_2}\cdots \mathfrak{N}_{\mathfrak{a}_m} = n$. The Dirichlet series associated to the arithmetic function $ v_{\K}^m(n)$ is 
$$
\sum_{n=1}^\infty \frac{v_{\K}^m(n)}{n^w} = \zeta_\K(w)^m,
$$
 which naturally arises from the product of $m$ Dedekind zeta functions. We apply the generalized Perron's formula \eqref{Perron's formula} on $I_\K^{m, \rho}(x)$ and obtain
\begin{equation}\label{Generalized Perron}
I_\K^{m, \rho}(x) = \frac{1}{2\pi i}\int_{C - i\infty}^{C + i\infty}f(w){\rm d}w,
\end{equation}
where $C > 1$ and 
$$
f(w) = \frac{\Gamma(w)}{\Gamma(w+\rho+1)}\zeta_\K(w)^m x^{\rho+w}= \frac{1}{w(w+1)\cdots (w+\rho)}\zeta_\K(w)^m x^{\rho+w}.
$$
We consider the contour $\mathcal{C}$ given by the rectangle with vertices $\{C - iT, C + iT, \mu +iT, \mu - iT \}$ in the anticlockwise direction as $T \to \infty$. The integrand $f(w)$ is analytic inside the contour $\mathcal{C}$ except for simple poles at $w=0$ and $w = 1$. The residues of $f(w)$ at $w = 0$ and $w=1$ are
\begin{equation*}
\underset{w=0} \Res \, \ f(w) = \zeta_\K(0)^m\frac{x^{\rho}}{\rho !}
\end{equation*}
and
\begin{equation*}
\underset{w=1} \Res \, \ f(w) = \(\underset{w=1} \Res \, \ \zeta_\K(w)^m \)\frac{x^{\rho+1}}{(\rho+1) !} = H_m \frac{x^{\rho+1}}{(\rho+1) !}
\end{equation*}
respectively, where $H_m = \underset{w=1} \Res \, \ \zeta_\K(w)^m$. Hence by Cauchy's residue formula we have
\begin{equation}\label{Residue theorem}
\frac{1}{2\pi i}\int_{\mathcal{C}} f(w){\rm d}w =\underset{w=0} \Res \, \ f(w) + \underset{w=1}\Res \, \ f(w) =\zeta_\K(0)^m\frac{x^{\rho}}{\rho !} + H_m \frac{x^{\rho+1}}{(\rho+1) !}.
\end{equation}
Combining \eqref{Generalized Perron} and \eqref{Residue theorem}, it follows that
\begin{equation}\label{Horizontal and vertical}
I_\K^{m, \rho}(x) =\zeta_\K(0)^m\frac{x^{\rho}}{\rho !} + H_m \frac{x^{\rho+1}}{(\rho+1) !} + \mathcal{H}_1 + \mathcal{H}_2 + \mathcal{V}
\end{equation}
where $\mathcal{H}_1 = \underset{T \to \infty} \lim \frac{1}{2\pi i}\int_{\mu + iT}^{C+iT} f(w) {\rm d}w$ and $\mathcal{H}_2 = \underset{T \to \infty} \lim \frac{1}{2\pi i}\int_{C - iT}^{\mu - iT} f(w) {\rm d}w$ are the horizontal integrals and $\mathcal{V} = \frac{1}{2\pi i}\int_{\mu - i\infty}^{\mu + i\infty} f(w){\rm d}w$ is the vertical integral. Firstly, we want to estimate the horizontal integrals. We replace $w$ by $\sigma + iT$ and apply Lemma \ref{Convexity bound} respectively to obtain
\begin{align}
\bigg| \frac{1}{2\pi i}\int_{\mu + iT}^{C+iT} f(w) {\rm d}w\bigg| &\leq  \frac{1}{2\pi }\int_{\mu}^{C} |\zeta_\K (\sigma + it)|^m \frac{x^\sigma}{T^{\rho + 1}} {\rm d}\sigma \nonumber\\
& \ll \int_{\mu}^{C} (T^{md}D_\K^m)^{\frac{1-\sigma}{2}+\epsilon} \frac{x^\sigma}{T^{\rho +1}} {\rm d}\sigma \nonumber\\
& \ll T^{\frac{md}{2} + \epsilon - (\rho + 1)} D_\K^{\frac{m}{2} + \epsilon} \underset{\mu \leq \sigma \leq C} \max x^\sigma (T^{md}D_\K^m)^{- \frac{\sigma}{2}}\nonumber\\
& \leq T^{\frac{md}{2} + \epsilon - (\rho + 1)} D_\K^{\frac{m}{2} + \epsilon} \{x^C (T^{md}D_\K^m)^{- \frac{C}{2}} +     x^\mu (T^{md}D_\K^m)^{- \frac{\mu}{2}} \}\nonumber\\
&\leq T^{\frac{md}{2}(1 - \mu) + \epsilon - (\rho + 1)} D_\K^{\frac{m}{2}(1 - \mu) + \epsilon} \{x^C + x^\mu \}.
\end{align}
Now, from the assumption \eqref{Assumption order} on $\rho$, it satisfies that
\begin{equation}
\bigg| \frac{1}{2\pi i}\int_{\mu + iT}^{C+iT} f(w) {\rm d}w \bigg| \leq \frac{D_\K^{\frac{m}{2}(1 - \mu) + \epsilon}}{T^{2-\epsilon}}(x^C + x^\mu).
\end{equation}
Therefore, we can conclude that $\mathcal{H}_1$ vanishes as $T \to \infty$. Similarly, one can show that $\mathcal{H}_2$ vanishes as $T \to \infty$. 
We now shift our attention to the vertical integral $\mathcal{V}$. It follows from \eqref{functional equation} that
\begin{equation}\label{Functional equation 2 }
\zeta_\K(w) = D_\K^{1/2 - w}\(\frac{\pi^{w-1/2} \Gamma(\frac{1-w}{2})}{\Gamma(\frac{w}{2})} \)^{r_1} \(\frac{(2\pi)^{2w-1} \Gamma(1-w)}{\Gamma(w)} \)^{r_2} \zeta_\K (1-w).
\end{equation}
We apply \eqref{Reflection formula} and \eqref{Duplication formula} to obtain
\begin{align}\label{Gamma properties}
\frac{\Gamma(\frac{1-w}{2})^{r_1}}{\Gamma(\frac{w}{2})^{r_1}} 
= \frac{[\Gamma(\frac{1-w}{2})\Gamma(1 - \frac{w}{2})]^{r_1}}{[\Gamma(\frac{w}{2})\Gamma(1 - \frac{w}{2})]^{r_1}}
=  \frac{[\Gamma(\frac{1-w}{2})\Gamma(\frac{1}{2} + \frac{1 - w}{2})]^{r_1}}{\(\frac{\pi}{\sin \frac{\pi}{2}w} \)^{r_1}}
= \( \frac{2^w}{\sqrt{\pi}} \)^{r_1} \(\sin \frac{\pi}{2}w \)^{r_1} \Gamma(1-w)^{r_1}.
\end{align}
Inserting \eqref{Gamma properties} into \eqref{Functional equation 2 }, we get
\begin{align*}
\zeta_\K(w) &= D_{\K}^{1/2 - w}i^{r_1}(2\pi)^{dw - r_1 - r_2} \(e^{-\frac{i\pi w}{2}} - e^{\frac{i\pi w}{2}} \)^{r_1} \frac{\Gamma(1-w)^{r_1+r_2}}{\Gamma(w)^{r_2}}\zeta_\K(1-w) \\
&= D_\K^{1/2 - w}i^{r_1}(2\pi)^{dw - r_1 - r_2} \sum_{j=0}^{r_1} (-1)^j 
\( \begin{matrix}
r_1 \\ j
\end{matrix} \) 
\( e^{-\frac{i\pi w}{2}} \)^{r_1 - j} \( e^{\frac{i\pi w}{2}} \)^j \frac{\Gamma(1-w)^{r_1+r_2}}{\Gamma(w)^{r_2}}\zeta_\K(1-w).
\end{align*}
Therefore, taking the $m$-th power, we have 
\begin{equation}\label{Equation 3.10}
\zeta_\K (w)^m = \frac{i^{mr_1}D_\K^{m/2}}{(2\pi)^{mr_1 + mr_2}} \sum_{j=0}^{mr_1} (-1)^j \( \begin{matrix}
mr_1 \\ j
\end{matrix} \) 
\[ (e^{\frac{i\pi}{2}})^{2j-mr_1}\]^w \( \frac{(2\pi)^{dm}}{D_{\K}^m} \)^w \frac{\Gamma(1-w)^{mr_1+mr_2}}{\Gamma(w)^{mr_2}}\zeta_\K(1-w)^m.
\end{equation}
We now insert \eqref{Equation 3.10} into the integrand of the vertical integral and obtain
\begin{align}\label{Equation 3.11}
\mathcal{V} = \frac{i^{mr_1}D_\K^{m/2}}{(2\pi)^{mr_1 + mr_2}} \sum_{j=0}^{mr_1} (-1)^j \( \begin{matrix}
mr_1 \\ j
\end{matrix} \) 
\frac{1}{2\pi i} \int_{\mu - i\infty}^{\mu + i\infty}
&\[ \(e^{\frac{i\pi}{2}}\)^{2j-mr_1} \frac{(2\pi)^{dm}}{D_{\K}^m} \]^w \frac{\Gamma(1-w)^{mr_1+mr_2}}{\Gamma(w)^{mr_2}} \nonumber\\
&\times \zeta_\K(1-w)^m \frac{1}{w(w+1)\cdots (w+\rho)} x^{\rho+w} {\rm d}w \nonumber\\
= \frac{i^{mr_1}D_{\K}^{m/2}}{(2\pi)^{mr_1 + mr_2}} \sum_{j=0}^{mr_1} (-1)^j \( \begin{matrix}
mr_1 \\ j
\end{matrix} \) 
\frac{1}{2\pi i} \int_{\mu - i\infty}^{\mu + i\infty}
&\[ \(e^{\frac{i\pi}{2}}\)^{2j-mr_1} \frac{(2\pi)^{dm}}{D_{\K}^m} \]^w \frac{\Gamma(1-w)^{mr_1+mr_2}}{\Gamma(w)^{mr_2}} \nonumber\\
&\times \sum_{n=1}^\infty \frac{v_{\K}^m(n)}{n^{1-w}} \frac{1}{w(w+1)\cdots (w+\rho)} x^{\rho+w} {\rm d}w. 
\end{align} 
It follows from the remark in \S \ref{Riesz sum} that we can interchange the integral and the summation in \eqref{Equation 3.11} under our assumption $\rho > 1$.  Hence we have 
\begin{align*}
\mathcal{V} = \frac{i^{mr_1}D_{\K}^{m/2}}{(2\pi)^{mr_1 + mr_2}} \sum_{j=0}^{mr_1} (-1)^j \( \begin{matrix}
mr_1 \\ j
\end{matrix} \) 
&\sum_{n=1}^\infty \frac{v_{\K}^m(n)}{n}
\frac{1}{2\pi i} \int_{\mu - i\infty}^{\mu + i\infty}
\[ \(e^{\frac{i\pi}{2}}\)^{2j-mr_1} \frac{(2\pi)^{dm}}{D_{\K}^m}n \]^w  \\
&\times \frac{\Gamma(1-w)^{mr_1+mr_2}}{\Gamma(w)^{mr_2}}  \frac{1}{w(w+1)\cdots (w+\rho)} x^{\rho+w} {\rm d}w. 
\end{align*}
We now differentiate the vertical integral  $\rho$-times with respect to $x$ to obtain
\begin{align}
\frac{d^\rho}{dx^\rho}\mathcal{V} = \frac{i^{mr_1}D_{\K}^{m/2}}{(2\pi)^{mr_1 + mr_2}} \sum_{j=0}^{mr_1} (-1)^j \( \begin{matrix}
mr_1 \\ j
\end{matrix} \) 
\sum_{n=1}^\infty \frac{v_{\K}^m(n)}{n}
\frac{1}{2\pi i} &\int_{\mu - i\infty}^{\mu + i\infty}
\[ \(e^{\frac{i\pi}{2}}\)^{2j-mr_1} \frac{(2\pi)^{dm}}{D_{\K}^m}n \]^w  \nonumber\\
&\times \frac{\Gamma(1-w)^{mr_1+mr_2}}{\Gamma(w)^{mr_2}} \, \ \frac{x^w}{w} \, \ {\rm d}w.
\end{align}
It follows from the functional equation \eqref{Gamma functional equation} for the
$\Gamma$-function that $\Gamma(1-w) = -w\Gamma (-w)$. Hence we have
\begin{align}\label{Equation 3.13}
\frac{d^\rho}{dx^\rho}\mathcal{V} = - \frac{i^{mr_1}D_{\K}^{m/2}}{(2\pi)^{mr_1 + mr_2}} \sum_{j=0}^{mr_1} (-1)^j \( \begin{matrix}
mr_1 \\ j
\end{matrix} \) 
\sum_{n=1}^\infty \frac{v_{\K}^m(n)}{n}
\frac{1}{2\pi i} &\int_{\mu - i\infty}^{\mu + i\infty}
\[ \(e^{\frac{i\pi}{2}}\)^{2j-mr_1} \frac{(2\pi)^{dm}}{D_{\K}^m}nx \]^w \nonumber \\
&\times \frac{\Gamma(1-w)^{mr_1+mr_2-1}\Gamma(-w)}{\Gamma(w)^{mr_2}} \, \ {\rm d}w.
\end{align}
As before, we let $\boldsymbol{1}_\ell$ denote the $\ell$-tuple all of whose entries equal $1$. We plug the definition of the Meijer G-function \eqref{G-function} into \eqref{Equation 3.13} and  obtain
\begin{align}\label{Equation 3.14}
\frac{d^\rho}{dx^\rho}\mathcal{V} = - \frac{i^{mr_1}D_{\K}^{m/2}}{(2\pi)^{mr_1 + mr_2}} &\sum_{j=0}^{mr_1} (-1)^j \( \begin{matrix}
mr_1 \\ j
\end{matrix} \) 
\sum_{n=1}^\infty \frac{v_{\K}^m(n)}{n}\nonumber\\
& \times G_{0, \,\ md}^{m(r_1+r_2), 0}\(\begin{matrix}
- \\  \boldsymbol{1}_{mr_1+mr_2-1}, 0, \boldsymbol{1}_{mr_2}
\end{matrix}
\bigg| \frac{(2\pi)^{dm}}{D_{\K}^m} (e^{\frac{i\pi}{2}})^{2j-mr_1}nx\).
\end{align}
Differentiating both sides of \eqref{Horizontal and vertical} $\rho$-times with respect to $x$, we obtain
\begin{equation}
\frac{d^\rho}{dx^\rho}I_\K^{m, \rho}(x) = \zeta_\K(0)^m + H_m x + \frac{d^\rho}{dx^\rho}\mathcal{V}.
\end{equation}
Finally, we obtain the result  
\begin{align}
I_\K^m(x) = \zeta_\K(0)^m + H_m x &- \frac{i^{mr_1}D_{\K}^{m/2}}{(2\pi)^{mr_1+mr_2}}\sum\limits_{j=0}^{mr_1}(-1)^j                \(\begin{matrix}
mr_1\\ j
\end{matrix}\)
\sum\limits_{n=1}^\infty \frac{v_k^m(n)}{n} \nonumber\\ 
&\times G_{0, \,\ md}^{m(r_1+r_2), 0}\(\begin{matrix}
- \\  \boldsymbol{1}_{mr_1+mr_2-1}, 0, \boldsymbol{1}_{mr_2}
\end{matrix}
\bigg| \frac{(2\pi)^{dm}}{D_{\K}^m} (e^{\frac{i\pi}{2}})^{2j-mr_1}nx\),
\end{align}
using Lemma \ref{Lemma Riesz sum} and \eqref{Equation 3.14}. This completes the proof of Theorem \ref{Theorem 1}.

\subsection{Proof of corollaries :}
Corollary \ref{Corollary 1.2} follows from Theorem \ref{Theorem 1} by considering $r_1 = d$ and $r_2 = 0$ in Theorem \ref{Theorem 1} where  $(r_1, r_2)$ is the signature of the number field $\K$. Here the error terms can be obtained from the relation \eqref{Voronoi function}. Corollary \ref{Corollary 1.3} also follows from Theorem \ref{Theorem 1} by considering $r_1 = 0$ and $r_2 = d/2$ in Theorem \ref{Theorem 1}.

\section{Applications}
In this section, we investigate some special cases of Theorem \ref{Theorem 1}.
\subsection{Piltz divisor problem in $\mathbb{Q}$}
We consider first the problem of estimating the partial sum 
$$
I_{\Q}^m(x) = \sideset{}{'} \sum_{n\leq x} d_m(n)
$$
where $\sum'$ means that the term corresponding to $n=x$ is halved and $d_m(n)$ counts the number of ways that $n$ can be written as a product of $m$ numbers. This problem can be considered a special case of Corollary \ref{Corollary 1.2} by taking
$d=1$ in Corollary \ref{Corollary 1.2}. We can conclude the following.
\begin{theorem}\label{Theorem 2}
For every $x> 0$, we have
\begin{align}
\sideset{}{'}\sum_{n\leq x} d_m(n) =xP_{m-1}(\log x)  + \(-\frac{1}{2}\)^m &- \frac{i^{m}}{(2\pi)^{m}}\sum\limits_{j=0}^{m}(-1)^j \(\begin{matrix}
m\\ j
\end{matrix}\)
\sum\limits_{n=1}^\infty \frac{d_m(n)}{n}  \nonumber\\
 &\times V\( (e^{\frac{i\pi}{2}})^{2j-m} (2\pi)^{m} nx ; \boldsymbol{1}_{m-1}, 0 \),
\end{align}
where $P_{m-1}(t)$ is a polynomial of degree $m-1$ in $t$ such that the coefficients can be evaluated from the relation
$$P_{m-1}(\log x) = \underset{s=1}\Res \, \ \zeta^m(s)\frac{x^{s-1}}{s}.$$
\end{theorem}
Here the main term of the partial sum $I_\Q^m(x)$ can be obtained from the sum of the 
residues of the function $\zeta^m(s)x^s s^{-1}$ at $s=0$ and $s=1$. We obtain Vorono\"i's theorem from Theorem \ref{Theorem 2} by considering $m=2$. Let $d_2(n) = d(n)$ be the divisor function. Then we can conclude the following as a corollary.
\begin{corollary}[Vorono\"i's Theorem]\label{Voronoi proof}
For every $x > 0$, we have
\begin{equation*}
\sideset{}{'}\sum_{n\le x}\!\!d(n) =x\log x+(2\gamma-1)x+\frac{1}{4} - \sum_{n=1}^\infty\frac{d(n)}{n}
\left(Y_1\left(4\pi\,\sqrt{xn}\,\right)+\frac{2}{\pi}K_1\left(4\pi\,\sqrt{xn}\,\right)\right)\sqrt{xn},
\end{equation*}
where $\gamma$ is Euler-Mascheroni constant. 
\end{corollary}
\begin{proof}
The main term of the partial sum $I_\Q^2(x)$ is basically the sum of the residues of $\zeta^2(s)x^s s^{-1}$ at $s=0$ and $s=1$. Let $\Delta_\Q^2(x)$ denote the error term of $I_\Q^2(x)$. It follows from Theorem \ref{Theorem 2} that 
\begin{align*}
\Delta_\Q^2(x) = \sum_{n=1}^\infty \frac{d(n)}{n}\frac{1}{4\pi^2}\[ V(e^{-i\pi} 4\pi^2nx ; 1,0) -2V(4\pi^2nx ; 1,0) + V(e^{i\pi} 4\pi^2nx ; 1,0) \].
\end{align*}
We apply \eqref{G and K function} here and obtain
\begin{align}\label{Error term with K}
\Delta_\Q^2(x) =\sum_{n=1}^\infty \frac{d(n)}{n} \frac{\sqrt{nx}}{\pi}\[ -i K_1(- 4\pi i \sqrt{nx}) -2K_1(4\pi \sqrt{nx}) + i K_1(4\pi i \sqrt{nx}) \].
\end{align}
The Bessel functions are inter-connected via the following relations (cf. \cite{BatemanE}).
\begin{equation}\label{Y I and K functions}
Y_\nu(iz) = e^{\frac{\pi i (\nu +1)}{2}}I_\nu(z) - \frac{2}{\pi} e^{-\frac{\pi i \nu}{2}}K_\nu(z),
\end{equation}
where $-\pi < \arg z \leq \frac{\pi}{2}$ and
\begin{equation}\label{J and I functions}
J_\nu(iz) = e^{\frac{\pi i \nu}{2}} I_\nu(z).
\end{equation}
It follows from \eqref{Y I and K functions} and \eqref{J and I functions} that for $\nu = 1$, we have
\begin{equation}\label{K Y and J result}
K_1(iz) = -\frac{\pi}{2} [J_1(-z) + i Y_1(-z)].
\end{equation}
We also need the following two basic formulas to evaluate Bessel functions of integer order at negative arguments:
\begin{equation}\label{J at negative arg}
J_n(-z) = (-1)^n J_n(z)
\end{equation}
and
\begin{equation}\label{Y at negative arg}
Y_n(−z) = (-1)^n Y_n(z) + 2i (-1)^n  J_n(z).
\end{equation}
We can now conclude the desired result inserting \eqref{K Y and J result}, \eqref{J at negative arg} and \eqref{Y at negative arg} into \eqref{Error term with K}.
\end{proof}

\subsection{Ideal counting problem} 
We now consider the problem of counting the number of ideals $\mathfrak{a}$ in any number field $\K$ such that the norm of an ideal $\mathfrak{N}_\mathfrak{a} \leq x$. This problem can be considered as a special case of Theorem \ref{Theorem 1} by considering $m=1$ in Theorem \ref{Theorem 1}. We substitute $I_\K(x), H$ in place of $I_\K^1(x), H_1$ respectively for simplicity. We can conclude the following.

\begin{theorem}\label{Theorem 4.3}
Let $\K$ be any number field of degree $d$ with signature $(r_1, r_2)$ and discriminant $d_\K$. Then we have
\begin{align}\label{Eqn 4.8}
I_\K(x) =\zeta_\K(0) + Hx &- \frac{i^{r_1}D_\K^{1/2}}{(2\pi)^{r_1+r_2}}\sum\limits_{j=0}^{r_1}(-1)^j                \(\begin{matrix}
r_1\\ j
\end{matrix}\)
\sum\limits_{n=1}^\infty \frac{v_\K(n)}{n} \nonumber\\ 
&\times G_{0, \,\ d}^{r_1+r_2, 0}\(\begin{matrix}
- \\  \boldsymbol{1}_{r_1+r_2-1}, 0, \boldsymbol{1}_{r_2}
\end{matrix}
\bigg| (e^{\frac{i\pi}{2}})^{2j-r_1} \frac{(2\pi)^{d}}{D_{\K}} nx\).
\end{align}
\end{theorem}
The following theorem provides the result for the partial sum $I_\K(x)$ when $\K$ is totally real. 
\begin{theorem}\label{Theorem 4.4}
Let $\K$ be any totally real number field of degree $d \geq 2$ with discriminant $d_\K$. Then we have
\begin{align}\label{Eqn 4.9}
I_\K(x) = H x - \frac{i^{d}D_{\K}^{1/2}}{(2\pi)^{d}}\sum\limits_{j=0}^{d}(-1)^j \(\begin{matrix}
d\\ j
\end{matrix}\)
\sum\limits_{n=1}^\infty \frac{v_\K(n)}{n}  
 \, V\( (e^{\frac{i\pi}{2}})^{2j-d} \frac{(2\pi)^{d}}{D_{\K}} nx ; \boldsymbol{1}_{d-1}, 0 \).
\end{align}
\end{theorem}
\begin{proof}
The proof of the theorem follows immediately from Theorem \ref{Theorem 4.3} by replacing $r_1$ by $d$ and $r_2$ by $0$ in \eqref{Eqn 4.8}. It follows from proposition \ref{Prop 2.2} that for any number field of degree $d$ with signature $(r_1 , r_2)$, Dedekind zeta function vanishes at $0$ when $r_1 + r_2 >1$. Hence the term $\zeta_\K(0)$ does not appear in the formula for considering the number field with $r_1 \geq 2$ and $r_2 = 0$. 
\end{proof}
We next consider the result for a real quadratic field as a corollary of Theorem \ref{Theorem 4.4} and estimate the partial sum $I_\K(x)$. 
\begin{corollary}
Let $\K$ be any real quadratic field with discriminant $d_\K$. Then we have
\begin{align}
I_\K(x) = H x - \sum_{n=1}^\infty\frac{v_\K(n)}{n}
\left[Y_1\left(\frac{4\pi\,\sqrt{xn}}{D_\K^{1/2}}\,\right)+\frac{2}{\pi}K_1\left(\frac{4\pi\,\sqrt{xn}}{D_\K^{1/2}}\) \right]\sqrt{xn}.
\end{align}
\end{corollary}
\begin{proof}
The proof follows an argument similar to that given in the proof of the Corollary \ref{Voronoi proof}. 
\end{proof}
The next theorem provides the result for the partial sum $I_\K(x)$ when $\K$ is purely imaginary which follows directly from Theorem \ref{Theorem 4.3}.
\begin{theorem}\label{Theorem 4.6}
Let $\K$ be any purely imaginary number field of degree $d$ with discriminant $d_\K$. Then
\begin{align}
I_\K(x) =\zeta_\K(0) + H x - \(\frac{D_\K}{(2\pi)^d}\)^{1/2}
\sum\limits_{n=1}^\infty \frac{v_\K(n)}{n} 
G_{0, \, d}^{\frac{d}{2}, \, 0}\(\begin{matrix}
- \\  \boldsymbol{1}_{\frac{d}{2}-1}, 0, \boldsymbol{1}_{\frac{d}{2}}
\end{matrix}
\bigg| \frac{(2\pi)^{d}}{D_{\K}} nx\).
\end{align}
\end{theorem}
We finally consider the special case of Theorem \ref{Theorem 4.6} for imaginary quadratic fields and estimate the partial sum $I_\K(x)$.
\begin{corollary}
Let $\K$ be any imaginary quadratic field with discriminant $d_\K$. Then we have
\begin{align}
I_\K(x) =\zeta_\K(0) + H x + \sum_{n=1}^\infty \frac{v_\K(n)}{n} J_1\(\frac{4\pi\sqrt{nx}}{D_\K^{1/2}}\)\sqrt{nx}.
\end{align}
\end{corollary}
\begin{proof}
It follows from Theorem \ref{Theorem 4.6} that since $\K$ is a purely imaginary number field of degree $2$, we have
\begin{align}
I_\K(x) = \zeta_\K(0) + H x - \frac{D_\K^{1/2}}{2\pi}
\sum\limits_{n=1}^\infty \frac{v_\K(n)}{n} 
G_{0, \, 2}^{1, \, 0}\(\begin{matrix}
- \\  0, 1
\end{matrix}
\bigg| \frac{4\pi^2}{D_{\K}} nx\).
\end{align}
We now apply \eqref{G and J function} to conclude our result.
\end{proof}
\section*{Acknowledgements}
The author would like to thank Prof. Ben Kane for helpful comments and discussions about this manuscript. The author would also like to express his gratitude to anonymous referee for the valuable comments about the manuscript.


\begin{thebibliography}{99}

\bibitem{Ayoub} R. G. Ayoub, {\em An introduction to the analytic theory of numbers}, American Math. Soc., (1963).

\bibitem{Bateman} H. Bateman, A. Erd\'elyi, {\em Higher Transcendental Functions}, {\bf Vol. I}, New York: McGraw-Hill, (1953).

\bibitem{BatemanE} H. Bateman, A. Erd\'elyi, {\em Higher Transcendental Functions}, {\bf Vol. II}, New York: McGraw-Hill, (1953).

\bibitem{Bourgain} J. Bourgain and N. Watt, {\em Mean square of zeta function, circle problem and divisor problem revisited}, Preprint, (2017), https://arxiv.org/abs/1709.04340. 

\bibitem{CM} K. Chandrasekharan and S. Minakshisundaram, {\em Typical means}, Oxford Univ. Press, Oxford, (1952).

\bibitem{Davenport} H. Davenport, {\em Multiplicative number theory}, 1st ed. Markham, Chicago (1967), 2nd ed. Springer, New York etc. (1980).

\bibitem{Hardy} G. H. Hardy and M. Riesz, {\em The general theory of Dirichlet's series}, Cambridge University Press, Cambridge (1915).

\bibitem{Iwaniec} H. Iwaniec and E. Kowalski, {\em Analytic Number Theory}, Amer. Math. Soc. Colloquium Publ. {\bf 53}, Amer. Math. Soc., Providence, RI, (2004).

\bibitem{KT} S. Kanemitsu and H. Tsukada, {\em Contributions to the theory of zeta-functions: the modular relation supremacy}, World Scientific, (2015).

\bibitem{Krishnarjun} Krishnarjun K., {\em On the number of  integral ideal of a number field}, Preprint, (2020), https://arxiv.org/abs/2002.06342.

\bibitem{Lang} S. Lang, {\em Algebraic number theory}, Addison-Wesley: Reading, MA, (1970).

\bibitem{Nowak} W. G. Nowak,{\em On the distribution of integral ideals in algebraic number theory fields}, Math. Nachr., {\bf 161}, (1993) 59--74.

\bibitem{Steen} S. W. P. Steen, {\em Divisor functions: their differential equations and recurrence formulae}, Proc. London Math. Soc. (2), {\bf 31} (1930), 47--80.

\bibitem{Takeda} W. Takeda, {\em Uniform bounds on the Piltz divisor problem over number fields}, Pacific J. Math. \textbf{301} (2019), 601--616.

\bibitem{Voronoi} G. Vorono\"i, {\em Sur une fonction transcendante et ses applications \`a la sommation de quelques s\'eries}, Ann. Ecole Norm. Sup., {\bf 21} (1904), 207--267, 459--533.

\bibitem{Wang} N. Wang, S. Banerjee, {\em On the product of Hurwitz zeta-functions}, Proc. Japan Acad. Ser. A (Math. Sci.), {\bf 93 (5)}, (2017), 31--36.

\end{thebibliography}
\end{document}